\documentclass{amsart}
\usepackage{graphicx}
\usepackage{amssymb}
\newtheorem{thm}{Theorem}[section]

\newtheorem{lemma}[thm]{Lemma}

\newtheorem{prop}[thm]{Proposition}
\theoremstyle{definition}
\newtheorem{defn}[thm]{Definition}
\theoremstyle{remark}
\newtheorem{rem}[thm]{Remark}
\theoremstyle{question}
\newtheorem{que}[thm]{Question}
\theoremstyle{Conjecture}
\newtheorem{con}[thm]{Conjecture}

\numberwithin{equation}{section}

\begin{document}

\title[A counterexample]{groups with the same number of centralizers}%
\author{K. Khoramshahi and M.  Zarrin}%

\address{Department of Mathematics, University of Kurdistan, P.O. Box: 416, Sanandaj, Iran}%
 \email{ M.Zarrin@uok.ac.ir and kh.khoamshahi@sci.uok.ac.ir}
\begin{abstract}
 For any group $G$, let  $nacent(G)$ denote the set of all nonabelian centralizers of $G$.   Amiri and Rostami in (Publ. Math. Debrecen
87/3-4 (2015), 429-437) put forward  the following question:
   Let H and G be finite simple groups. Is it true that if $|nacent(H)| = |nacent(G)|$, then $G$ is isomorphic to $H$? 
  In this paper, among other things, we give a negative answer to this question.
\end{abstract}
\maketitle

\section{\textbf{ Introduction}}

For any group $G$, let $cent(G)$ denote the set of centralizers of $G$ and $nacent(G)$ denote the set of all non-abelian centralizers belonging to $cent(G)$.  We say that a group $G$ (not necessarily finite group) has $n$ centralizers (or $G$ is a $\mathcal{C}_n$-group) if $|cent(G)| = n$.  It is clear that a group is a $\mathcal{C}_1$-group if and only if it is abelian. 
The class of finite $\mathcal{C}_n$-groups was introduced by Belcastro and Sherman in \cite{bs} and investigated by many authors. For instance, see \cite{amr, zar1} for finite $\mathcal{C}_n$-groups and  \cite{zar4} for infinite $\mathcal{C}_n$-groups.

In 2005,  Ashrafi and Taeri  in \cite{ta}, because of the influence of $|cent(G)|$ on the structure of groups, raised the following question:  Let $G$ and $H$ be finite simple groups. Is it true that if $|cent(G)| = |cent(H)|$, then $G$ is isomorphic to $H$?
Zarrin in \cite {zar} disproved their question  with a counterexample. \\
We say that a group $G$ is a $CA$-group (or an $AC$-group) if the centralizer of every non-central element is abelian.  Therefore a group $G$ is $AC$-group if and only if  $|nacent(G)| = 1$. The authors in \cite{ar}, characterized all groups $G$ 
with  $|nacent(G)| = 2$ and finally raised the following question (see Question 2.13 of \cite{ar}):

\begin{que}\label{q2}
Let $G$ and $H$ be finite simple groups. Is it true that if $|nacent(G)| = |nacent(H)|$, then $G$ is isomorphic to $H$?
\end{que}
In section 2, we give a negative answer to this question.\\

It is easy to see that two simple groups are isomorphic if and only if they are isoclonic.
 Zarrin in \cite{zar3} proved that for every two isoclinic groups $G$ and $S$, $|cent(G)|=|cent(S)|$. The natural question is whether  the converse of his statement is true? 

\begin{que}\label{q3}
Let $G$ and $S$ be arbitrary groups. Is it true that if $|cent(G)|=|cent(S)|$, then $G$ and $S$ are isoclinic groups? 
\end{que}
  It is easy to see that this is not generally true. For example, the second author in \cite{zar} has  proved that,  $|cent(P SL(2, 23))| = |cent(A7)| = 807$ but $PSL(2, 23)\not \cong A_7$ and so they are not isoclinic groups.  Also let $D_{40}$ and $A_5$ be the dihedral group of degree $40$ and alternating group of degree $5$, respectively. It is not hard to see that $|cent(D_{40})|=|cent(A_5)|=22$; obviously they are not isoclinic.\\
Because of the importance of Question \ref{q3}, we are looking for special cases. For instance, in section 3, we will investigate Question \ref{q3} when $H$ is a subgroup of $G$. In fact, we conjecture that Question \ref{q3} would be true  under some  circumstances (see also Conjecture \ref{c}).\\

In the last section, we show that the derived length of a nilpotent $C_n$-group is at most $$2+[log^{n}_p]-\sum_{i=1}^{m-1}[log^{p_i+1}_p] ,$$ where $p=\min\{p_i \mid \text{G has non-abelian Sylow } p_i\text{-subgroup}\}$ and $m$ is the number of non-abelian Sylow $p_i$-subgroups of $G$ (this improve the main result in \cite{zar2}).

\section{\textbf{Counterexample to Question \ref{q2} }}

Obviously this question is not true for $|nacent(G)| = |nacent(H)|=1$.  In fact, if $G$ and $H$ are two simple $AC$-groups, then $|nacent(G)| = |nacent(H)|=1$ and  they are not necessarily isomorphic. For instance,  $PSL(2, 2^m)$ and $PSL(2,2^n)$ for $m\neq n\geq 3$ (note that if $q > 5$ and $q \equiv ~0 ~mod ~4$, then these groups are an $AC$-group).

Now we show that the above question is not true even if we have  $|nacent(G)| = |nacent(H)| \geq 2$, where $G$ and $H$ are two finite simple groups. For this, we need the following lemma. (In fact, finding $nacent(G)$ of a group itself is of independent interest as a pure combinatorial problem.)

\begin{lemma}
Let $G = P SL(2, q)$, where $q>5$ is a $p$-power {\rm(}$p$ prime{\rm)}. Then

\begin{itemize}
\item[(I)] If $q\equiv  0 ~mod~ 4$. Then $|nacent(G)|=1$.
\item[(II)] If $q\equiv 1~mod ~4$. Then $|naent(G)| = (q^2 + q + 2)/2$.
\item[(III)] If $q\equiv 2 ~mod ~4$. Then  $|naent(G)| = (q^2 -q + 2)/2$.

\end{itemize}

\end{lemma}

\begin{proof}
Case (I). It is clearly,  as $G$ is an $AC$-group. \\
Case (II).  In this case, according to  Lemma 3.21 of \cite{abd}, one can obtain that the number of abelian centralizers of $G$ is $q^2 +q+1$. Therefore, by Case (2) of Theorem 1.1 of \cite{zar},  we have $$|nacent(G)|=(3q^2 + 3q + 4)/ 2-(q^2 +q+1)=(q^2 + q + 2)/2.$$
Case (III). Similarly. 
\end{proof}

Now it is easy to see by GAP \cite{gap}, that for $PSU(3,3)$ (the projective special unitary group of degree 3 over the finite field of order 3)  we have $|naent(PSU(3,3))| =92$ and also by Case 2 of Lemma 1.4, $|nacent(P S L(2, 13)| =92$. Obviously, $PSU(3,3)\not \cong PSL(2,13)$ (in fact $|PSU(3,3)|\neq |PSL(2,3)|$ ).

\begin{rem}
It seems that, in view of the above questions, two indices $|cent(G)|$ and $|nacent(G)|$ have the same influence on the structure of groups. 
So the natural question that would be risen is that which one has stronger influence  on the other one? On the other hand, if there are  two groups, say $G$ and $H$ such that $|cent(G)|=|cent(H)|$, then does it guarantee $|nacent(G)|=|nacent(H)|$? (And vice versa?) The answer is no. For instance, see the following groups:
\begin{itemize}
\item[(1)] $|cent(A_7)|=|cent(PSL(2,23))|=807$, but $|nacent(A_7)|=141$ and \\
$|nacent(PSL(2,7))|=254$.
\item[(2)] $|nacent(A_5)|=|nacent(PSL(2,8))|=1$, but $|cent(A_5)|=22$ and \\
$|cent(PSL(2,8))|=74$. (Even though, in this case, there are lots of groups that satisfy this condition, such as every two simple $AC$-groups.)
\end{itemize}

\end{rem}

Finally, in view of the above counterexamples that have been mentioned for the questions in the introduction, we can pose the following conjecture: 
 
\begin{con}\label{c}
Let $G$ and $S$ be finite groups. Is it true that if $|cent(G)| = |cent(S)|$ and $|G'|=|S'|$, then $G$ is isoclonic to $S$?
\end{con}

\section{\textbf{Groups with the same number of centralizers}}

We start with the following definition.

\begin{defn}
We say that groups $G$ and $S$ are isoclinic, if there are isomorphisms

 $$\alpha:\frac{G}{Z(G)}\rightarrow \frac{S}{Z(S)} ~~ \text{and~~} \beta:G'\rightarrow S'$$
 such that for every $ g_1, g_2 \in G~~ \text{and~~}s_1, s_2 \in S $, if

 $$\alpha(g_1{Z(G)})= \ {s_1 Z(S) } ~~\text{and~~}   \alpha(g_2{Z(G)})= \ {s_2 Z(S)}$$
 then $\beta ([ g_1, g_2])=[ s_1, s_2]$. The pair $(\alpha ,\beta)$ is called a isoclinism from $G$ to $S$.
\end{defn}

First we give the following lemma.

\begin{lemma}\label{pt}
For a group $G$, we have the following statements:
\item[(1)]
If for every subgroup $H$ of group $G$, $|cent(H)|=|cent(G)|$ then $G$ is an abelian group.
\item [(2)]If for every non-abelian subgroup $H$ of group $G$, $|cent(H)|=|cent(G)|$ then $G$ is an $AC$-group {\rm (}clearly the converse is not true {\rm )}.
\end{lemma}
\begin{proof}
\item[(1)]
Clearly.
\item [(2)]
Assume that  $G$ is not an $AC$-group and $C_G (x)$ is a nonabelian subgroup, for some $x\in G\ Z(G)$.  Therefore $|cent(C_G (x))|=|cent(G)|$ and by Lemma $2.1$ of \cite{zar2}, it is a contradiction.
\end{proof}

We show that Question \ref{q3} is satisfied  for special maximal subgroups. For this, we need the following lemma. 
\begin{lemma}\label{pi}
Let $H$ be a subgroup of an arbitrary group $G$ such that $|cent(H)|=|cent(G)|$. Then $H\cap Z(G)=Z(H)$ and $\frac{H}{Z(H)}\cong \frac{HZ(G)}{Z(G)}$. In particular,  $H$ is isoclonic with $HZ(G)$.
\end{lemma}
\begin{proof}
First off, for the group $G$ and its subgroup $H$, we define set $cent_G(H)=\{C_G(h) \mid h\in H\}$, where $C_G (h)=\{x\in G \mid xh=hx\}$ is the centralizer of $h$ in $G$. Since for every $x\in G$, $C_G (x)\cap H=C_H (x)$, one can follow that $$|cent(H)|\leq |cent_G(H)|\leq |cent(G)|.$$ \\
Now suppose that $|cent(H)|=|cent(G)|=n$. Therefore  $|cent_G(H)|=|cent(G)|$. From which one can follow that  $Z(G)=\bigcap_{i=1}^{n}C_G (x_i)=\bigcap_{i=1}^{n}C_G (h_i)$ with $ x_i \in G, h_i \in H $. Therefore $Z(G)\cap H=\bigcap_{i=1}^{n}(C_G (x_i)\cap H)=\bigcap_{i=1}^{n}C_H (h_i)=Z(H)$.

As $\frac{HZ(G)}{Z(G)}\cong \frac{H}{Z(G)\cap H}=\frac{H}{Z(H)}$,  $H$ is isomorphic  with subgroup $HZ(G)$ of $G$ (note that $(HZ(G))'= H'$).
\end{proof}

We note that, if $H$ be a subgroup of group $G$, then $H\cap Z(G)\leq Z(H)$. In general, the converse is not true. For example, if $G$ be a centerless group and $H=\langle x \rangle$ be a cyclic subgroup of $G$. Then $H=Z(H)\neq Z(G)=\langle e \rangle$ and $Z(H)\nleq Z(G)\cap H$.

\begin{prop}\label{ok}
Let $M$ be a maximal subgroup of group $G$ such that $|cent(M)|=|cent(G)|$. Then  either $Z(M)=Z(G)$ or $M$ is isoclonic with $G$.
\end{prop}
\begin{proof}
Suppose that $M$ is a maximal subgroup of $G$. Therefore $M\leq MZ(G)\leq G$ and so either $M=MZ(G)$ or $MZ(G)=G$.

 If $M=MZ(G)$, then $Z(G)\leq M$ and so $Z(G)\cap M=Z(G)$. It follows, by Lemma \ref{pi}, that $Z(G)= Z(M)$.

If $MZ(G)=G$, then $G'= (MZ(G))'=M'$. On the other hand, by Lemma \ref{pi}, $\frac{M}{Z(M)}\cong \frac{MZ(G)}{Z(G)}=\frac{G}{Z(G)}$. Therefore $M$ and $G$ are isoclinic.
\end{proof}

In the next theorem we show that the question is satisfied for some $n$, where $|cent(G)|=n$.

\begin{thm}
Let $G$ be a non-abelian arbitrary group and $H\leq G$. If $|cent(H)|=|cent(G)|=n$, where $n<6$, then $H$ and $G$ are isoclinic groups.
\end{thm}
\begin{proof}
Since $|cent(H)|=|cent(G)|=n<6$, then, by Theorem 3.5 of \cite{zar3}, $\frac{G}{Z(G)}$ is isomorphic to $C_2\times C_2$, $S_3$ or $C_3\times C_3$. On the other hand,  by Lemma \ref{pi}, we have $\frac{H}{Z(H)}\cong \frac{HZ(G)}{Z(G)}\leq \frac{G}{Z(G)}$. If $\frac{HZ(G)}{Z(G)}<\frac{G}{Z(G)}$, since every subgroups of $C_2\times C_2$, $S_3$ and $C_3\times C_3$ are cyclic, $\frac{H}{Z(H)}$ is a cyclic group and whereby  $H$ is abelian, a contradiction. Therefore $\frac{H}{Z(H)}\cong \frac{HZ(G)}{Z(G)}=\frac{G}{Z(G)}$, so $HZ(G)=G$ and $G'= H'$. Thus $H$ and $G$ are isoclinic groups.
\end{proof}

Finally, we discuss a wide classes of groups in which $|cent(G)|=|cent(G')|$ and $G$ is isoclonic with $G'$.

\begin{prop}
Let $G$ be a finite group such that $\frac{G}{Z(G)}$ is isomorphic with a simple group. Then $G$ and $G'$ are isoclinic groups.
\end{prop}
\begin{proof}
As $\frac{G}{Z(G)}$ is isomorphic with a simple group, so we have $$\frac{G}{Z(G)}=({\frac{G}{Z(G)}})'=\frac{G'Z(G)}{Z(G)}\cong \frac{G'}{Z(G)\cap G'}.$$  It follows that $G'=G''$. For complete proving, it is enough to show that $Z(G)\cap G'=Z(G')$. It is clear that $Z(G)\cap G'\leq Z(G')$. Now if $Z(G')\nleq Z(G)$, then $\frac{Z(G')}{Z(G)\cap G'}$ is a normal subgroup of  $\frac{G}{Z(G)}\cong \frac{G'}{Z(G)\cap G'}$,  a contradiction. Thus $\frac{{G}}{Z(G)}=\frac{G'}{Z(G')}$, so $G$ and $G'$ are isoclinic groups.

\end{proof}

\section{\textbf{ The derived length of a nilpotent $C_n$-group}}

The author in \cite{zar4} showed that  the derived length of an arbitrary solvable $C_n$-group is at most $n$ and then he improved this bound for nilpotent groups. In fact, he shoved that, the derived length of an arbitrary nilpotent $C_n$-group is at most $[n-1/2]+1$, see \cite{zar2}. Here, with the same strategy,  we improve previous results for finite nilpotent $C_n$-groups, as follows:

\begin{thm}
Let $G$ be  a  finite non-abelian  nilpotent $C_n$-group. Then the derived length of $G$, say $d(G)$ is at most $2+[log^{n}_p]-\sum_{i=1}^{m-1}[log^{p_i+1}_p]$, where $p=\min\{p_i \mid \text{G has non-abelian Sylow } p_i\text{-subgroup}\}$ and $m$ is the number of non-abelian Sylow $p_i$-subgroups of $G$.
\end{thm}
\begin{proof}
First we suppose that $G$ is a finite $p$-group with $|cent(G)|=n$. Then in view of the proof of the main theorem of of \cite{zar4} and  Case(2) of Lemma 2.1 of \cite{zar4}, one can follow that the derived length of  $G$ is at most $[log^{n-1}_p]+2$.

Now we assume that $G$  is a  finite non-abelian nilpotent $C_n$-group. Therefore $G\cong \prod_{i=1}^{m}P_i$. As $d(G)=\max \{d(P_i)\mid 1\leq i\leq m\}$, we may assume that all Sylow $p_i$-subgroups of $G$ are non-abelian. Thus there exists $1\leq j \leq n$ such that $$d(G) \leq [log_{p_j}^{|cent(P_j)|-1}] +2 \leq [log_{p}^{|cent(P_j)|-1}]+2,$$  where $p$ is the smallest prime divisor of $|G|$. It is easy to see that $n=|cent(G)|=\prod_{i=1}^{i=m}|cent(P_i)|$. From this, and by Corollary 2.2 of \cite{zar2}, one can follow that $$|cent(P_j)|\leq n/(\prod_{j\neq i=1}^{m}(p_i+1))$$ and so  $d(G)\leq 2+[log^{n}_p]-\sum_{i=1}^{m-1}[log^{p_i+1}_p] $, as wanted.
\end{proof}

\end{document}